\documentclass[12pt]{amsart}

\usepackage{xy, graphicx, color,hyperref,verbatim,cleveref}
\pdfoutput=1
\usepackage{amsfonts, ytableau,array}
\usepackage[utf8]{inputenc}
\usepackage{amssymb,amsmath,amsthm,bm,mathrsfs,enumerate,bbm,enumitem}
\usepackage[a4paper, total={6in, 8in}]{geometry}
\usepackage[english]{babel}
\usepackage{fancyhdr}
\usepackage[utf8]{inputenc}
\usepackage{mathtools,cite,mathdots}
\usepackage{tikz}
\usepackage{tikz-cd}
\usetikzlibrary{matrix}
\usepackage{siunitx}
\usepackage{booktabs}
\usepackage{blindtext}
\usepackage{array}
\newcolumntype{P}[1]{>{\centering\arraybackslash}p{#1}}
\PassOptionsToPackage{linktocpage}{hyperref}
   \hypersetup{
           breaklinks=true,  
           colorlinks=true,  
           pdfusetitle=true, 
           citecolor={blue!80!black}
        }

\newenvironment{roster}
 {\begin{enumerate}[font=\upshape,label=(\roman*)]}
 {\end{enumerate}}

 \newtheorem{thm}{Theorem}[section]
\newtheorem{c.intro}[thm]{Corollary}
\newtheorem{lemma}[thm]{Lemma}
\newtheorem{prop}[thm]{Proposition}

\newtheorem{ex}[thm]{Example}
\theoremstyle{definition}
\newtheorem{remark}[thm]{Remark}
\theoremstyle{definition}
\newtheorem{defn}[thm]{Definition}

\newcommand{\nc}{\newcommand}

\nc{\mc}{\mathcal}
\nc{\mb}{\mathbb}
\nc{\mf}{\mathfrak}
\nc{\ul}{\underline}
\nc{\ol}{\overline}
\nc{\N}{\mb N}
\nc{\R}{\mb R}
\nc{\Z}{\mb Z}
\nc{\Q}{\mb Q}
\nc{\C}{\mb C}
\nc{\F}{\mb F}
\nc{\dmo}{\DeclareMathOperator}

\nc{\mat}[4]{
    \begin{pmatrix}
      #1 & #2 \\
      #3 & #4
    \end{pmatrix}
}
\dmo{\Ker}{Ker} \dmo{\val}{val} \dmo{\ord}{ord}
\dmo{\I}{I}
\dmo{\II}{II}
\dmo{\odd}{odd}
\dmo{\sgn}{sgn}
\dmo{\dett}{det}
\dmo{\Span}{Span}
\dmo{\Syl}{Syl}
\dmo{\diag}{diag}
\dmo{\Sq}{Sq}
\dmo{\irr}{Irr}
\nc{\beq}{\begin{equation*}}
\nc{\eeq}{\end{equation*}}
\nc{\half}{\frac{1}{2}}
\dmo{\Ima}{Im}

\dmo{\Mod}{mod}

\dmo{\core}{core}
\dmo{\res}{res}
\dmo{\lin}{lin}
 
 \dmo{\St}{St}
 \dmo{\st}{st}
 
 \dmo{\Tr}{Tr}
 \dmo{\RO}{RO}
\dmo{\Sp}{Sp}
\dmo{\SO}{SO}
\dmo{\SL}{SL}
\dmo{\GL}{GL}
 \dmo{\Spin}{Spin}
\dmo{\GSp}{GSp}
\nc{\la}{\lambda}
  \nc{\eps}{\varepsilon}
  
 \nc{\lip}{\langle}
 \nc{\rip}{\rangle}
\nc{\gm}{\gamma}
\dmo{\Top}{top}
\dmo{\Perm}{Perm}
\dmo{\Res}{Res}
\dmo{\Ind}{Ind}
\dmo{\ind}{ind}
\dmo{\tr}{tr}
\dmo{\Sym}{Sym}
\dmo{\reg}{reg}
\dmo{\End}{End}
\dmo{\Hom}{Hom}
\dmo{\Pin}{Pin}
\dmo{\Or}{O}
\dmo{\SW}{SW}
\dmo{\orb}{orb}
\title[Stiefel-Whitney Classes]{Stiefel-Whitney Classes of representations of $\SL(2,q)$} 
\author{Neha Malik}
\author{Steven Spallone}

\DeclareMathOperator{\pr}{pr}

\dmo{\ORep}{ORep}
\dmo{\Rep}{Rep}

\usepackage{blindtext,graphicx}
\usepackage[absolute]{textpos}
\newcommand{\as}{\alpha}

\newcommand{\tht}{\chi}

\newcommand{\f}{\mathbb{F}}
\newcommand{\z}{\mathbb{Z}}
\newcommand{\rr}{\mathbb{R}}
\newcommand{\cc}{\mathbb{C}}

\newcommand{\hh}{\mathbb{H}}

\newcommand{\fe}{\mathfrak{e}}

\address{Indian Institute of Science Education and Research, Pune-411008,Maharashtra,India}
\email{neha.malik@students.iiserpune.ac.in}
\address{Indian Institute of Science Education and Research, Pune-411008,Maharashtra,India}
\email{sspallone@gmail.com}
\date{\today}

\keywords{Stiefel-Whitney classes, special linear group, quaternions}
\subjclass{Primary 20G40, Secondary 55R40}

\begin{document}

\begin{abstract}
 We describe the Stiefel-Whitney classes (SWCs) of orthogonal representations $\pi$ of the finite special linear groups $G=\SL(2,\mb F_q)$, in terms of character values of $\pi$. 
From this calculation, we can answer interesting questions about SWCs of $\pi$. For instance, we determine the subalgebra of $H^*(G,\Z/2\Z)$  generated by the SWCs of orthogonal $\pi$, and we also determine which 
 $\pi$ have nontrivial mod $2$ Euler class.
 
 \end{abstract}
\maketitle
\tableofcontents
\section{Introduction}

The theory of Stiefel-Whitney classes (SWCs) of vector bundles is an old unifying concept in geometry.  Real representations, or equivalently, orthogonal complex representations of a group $G$ give rise to \emph{flat} vector bundles over the classifying space $BG$. Via this construction one associates SWCs $w_i(\pi)\in H^i(G,\z/2\z)$ to an orthogonal representation $\pi$ of $G$. Their sum $w(\pi)$, living in the group cohomology $H^*(G,\Z/2\Z)$, is known as the \textit{total Stiefel-Whitney class} of $\pi$.

 There do not seem to be many explicit calculations of such SWCs, notably when $G$ is nonabelian. (However see \cite{PG}.) This paper contains formulas for SWCs in terms of character values for $\SL(2,q)$, when $q$ is a prime power, achieved via \emph{detection results.} Formulas for SWCs of cyclic groups are well-known; we review this for the group of order $2$ in Section \ref{ea2}. The second SWCs of representations of $S_n$ and related groups were found in \cite{Ganguly}.

The case of $\GL(n,q)$ with $q$ odd was completed recently by Ganguly-Joshi in \cite{GJgl2} and \cite{GJgln}. To explain how their work relates to ours, we need to introduce ``detection''. 

Let $G$ be a finite group, and write $H^*(G,\z/2\z)$ for the mod $2$ group cohomology ring of $G$. Let $H^*_{\SW}(G,\z/2\z)$ be the subalgebra of $H^*(G,\z/2\z)$ generated by SWCs $w_i(\pi)$ of orthogonal representations $\pi$ of $G$. We say a subgroup $K$ \emph{detects} the mod $2$ cohomology of $G$ when the restriction map
$$i^*:H^*(G,\z/2\z)\to H^*(K,\z/2\z)$$
is an injection, and we say $K$ \textit{detects SWCs} of $G$, when this map is injective on $H^*_{\SW}(G,\z/2\z)$.
 Suppose one of these is true, and that $\pi$ is an orthogonal representation of $G$. Then $w(\pi)$ is specified by its image in $H^*(K,\z/2\z)$, which is actually the SWC of the restriction of $\pi$ to $K$. When there is a detecting subgroup, one simply writes $w(\pi)$ for this image. 
 
An essential tool in Ganguly-Joshi's papers is a theorem of Quillen  \cite{quillen}, asserting that the diagonal subgroup of $\GL(n,q)$ detects its mod $2$ cohomology.  However, such a result is not available to us. Instead, most of this paper is devoted to establishing the following:
 
\begin{thm}\label{prop} Let $G=\SL(2,q)$ with $q$ odd. Then the center $Z$ detects SWCs of $G$. 
\end{thm}
 
From this we obtain:
\begin{thm}\label{thm2.section} 
Let $G=\SL(2,q)$ with $q$ odd. Let $\pi$ be an orthogonal representation of $G$. 
Then the total SWC of $\pi$ is
\begin{equation*}
w(\pi)=(1+\mathfrak{e})^{r_\pi},
\end{equation*}
where $\mathfrak{e}$ is the non-zero element in $H^4(\SL(2,q),\z/2\z)$ and $r_\pi=\frac{1}{8}(\chi_\pi(\mathbbm{1})-\chi_\pi(-\mathbbm{1})).$ $($Here $\mathbbm{1}$ is the identity matrix.$)$ If $\pi$ is also irreducible, then $w(\pi)=1$.
\end{thm}  

In future work, these two theorems will be applied to determine SWCs of other finite groups of Lie type.

The case of $q$ even is quite different. Here the upper unitriangular subgroup $N \cong \mb F_q$ detects the mod $2$ cohomology of the special linear group, and we use this to prove:

\begin{thm} \label{prop.two} Let $(\pi,V)$ be an orthogonal representation of $G=\SL(2,q)$, with $q$ even. Then the total SWC of $\pi$ is
\beq
w(\pi)=(1+ \mc D  )^{m_\pi},
\eeq
where $m_\pi=\dim_{\C} \left( V/V^G \right)$, where $V^G$ is the $G$-fixed vectors in $V$, and $\mc D$ is the sum of the Dickson invariants of $\mb F_q$, viewed as a $\Z/2\Z$-vector space. 
\end{thm}

(Dickson invariants are reviewed in Section  \ref{even}.)

\bigskip

There are several corollaries of Theorems \ref{thm2.section}  and \ref{prop.two}. Let $G=\SL(2,q)$.  Firstly,  we describe the subgroup  of the complete cohomology ring $H^\bullet(G,\Z/2\Z)$ (see Section \ref{vr}) generated by $w(\pi)$, as $\pi$ varies over orthogonal representations. 

\begin{c.intro} 
This subgroup equals: 
\begin{roster}
\item $\{(1+\fe)^n: n\in\z\}$, for $q$ odd,
\item $\{(1+ \mc D  )^n : n \in \z\}$, for $q$ even.
\end{roster}
\end{c.intro}

Secondly, we compute $H^*_{\SW}(G,\z/2\z)$.
 
\begin{c.intro}
The subalgebra  $H^*_{\SW}(G,\z/2\z)$ is:
\begin{roster}
\item $\z/2\z[\fe]$, for $q$ odd,
\item generated by the Dickson invariants of $\f_q$, for $q$ even.
\end{roster}
\end{c.intro}

Another characteristic class of importance is the Euler class $e(\pi) \in H^d(G,\Z)$, where $d=\deg \pi$. We define $w_{\Top}(\pi)=w_{\deg \pi}(\pi)$; it is the ``top'' Stiefel-Whitney class of $\pi$. By Property 9.5 of \cite{milnor}, we know that $w_{\Top}(\pi)$ is the reduction of $e(\pi)$ mod $ 2$.

\begin{c.intro}\label{topec}
 Let $(\pi,V)$ be an orthogonal representation of $G$. Then $w_{\Top}(\pi) \neq 0$ precisely when:
\begin{roster}
\item\label{topodd} $\pi(-\mathbbm1)=-1$, meaning $\pi(-\mathbbm 1)$ acts on $V$ by the scalar $-1$, for $q$ odd,
\item $V^N=\{0\}$, for $q$ even. $($Equivalently, $\pi$ is cuspidal.$)$
\end{roster}
\end{c.intro}

Finally, we determine the \emph{obstruction degree} of an orthogonal $\pi$, meaning the least $k>0$ with $w_k(\pi)\neq 0$. (If $w(\pi)$ is trivial, then this degree is infinite.) Write $\ord_2(n)$ for the highest power of $2$ which divides an integer $n$.
\begin{c.intro} The obstruction degree of $\pi$ is:
\begin{roster}
\item $2^{t+2}$, where $t=\ord_2(r_\pi)$, for $q$ odd,
\item $2^{r+s-1}$, where $s=\ord_2(m_\pi)$ and $q=2^r$.
\end{roster}
\end{c.intro}

The layout of the paper is as follows: Section \ref{not} recalls some fundamental notions used throughout the paper.  Section  \ref{quatsec} is dedicated to the generalized quaternions, particularly the Quaternion group $Q$ of order 8. These groups are featured in the detection of $\SL(2,q)$ when $q$ is odd. Here we also determine $w_{4i}(\pi)$ for orthogonal representations $\pi$ of $Q$.  In Section \ref{odd}, we prove Theorems \ref{prop}  and \ref{thm2.section},  the corollaries for $q$ odd, and finally compare our result to the $\GL(2,q)$ case. In Section \ref{even}, we prove Theorem \ref{prop.two}  and the corollaries for $q$ even.

\bigskip

\textbf{Acknowledgements.} 
The authors would like to thank Ryan Vinroot, Dipendra Prasad, and Rohit Joshi for helpful discussions.  This paper comes out of the first author’s Ph.D. thesis \cite{Malik.thesis} at IISER Pune, during which she 
was supported by a Ph.D. fellowship from the Council of Scientific and Industrial Research, India.
 
\section{Notations \& Preliminaries}\label{not}

\subsection{Representations}
Let $G$ be a finite group. All our representations will be finite-dimensional. Let $\pi$ be a representation of $G$ on a complex vector space $V$. 
Write $\chi_\pi(g)$ for the character of $\pi$ at an element $g \in G$.
If $H$ is a subgroup of $G$, write $\res^G_H \pi$ for the restriction of $\pi$ to $H$.
Write $\reg(G)$ for the regular representation of $G$.

We say $\pi$  is \emph{orthogonal}, provided there exists a non-degenerate $G$-invariant symmetric bilinear form $B:V\times V\to \C$.
 We say it is \emph{symplectic}, provided there exists a non-degenerate $G$-invariant antisymmetric bilinear form $B:V\times V\to \C$.
 
When $\pi$ is self-dual and irreducible, it is either orthogonal or symplectic. In this case, the Frobenius-Schur Indicator $\eps(\pi)$ is a sign defined as $1$ when $\pi$ is orthogonal, and $-1$ when $\pi$ is symplectic.

\subsection{Orthogonal Representations}
We are especially interested in orthogonal representations.
 
 From \cite[Chapter II, Section 6]{BrokerDieck}, we note:

\begin{prop} \label{prop.one}
If $(\pi,V)$ is orthogonal, then there is a $(\pi_0,V_0)$, with $V_0$ a real vector space, so that $\pi_0 \otimes_{\R} \C \cong \pi$. Moreover, $\pi_0$ is unique up to isomorphism.
\end{prop}
 
Given a representation $(\pi,V)$ with dual $(\pi^\vee,V^\vee)$, we can symmetrize it to get an orthogonal representation of $G$ by defining $$S(\pi):=\pi\oplus\pi^\vee$$ on the vector space $V\oplus V^\vee$. We give a symmetric  $G$-invariant bilinear map $B$ on $V\oplus V^\vee$ as, $$B((v,\as),(w,\beta))= \langle \alpha,w \rangle+ \langle \beta,v \rangle.$$ 
We call $S(\pi)$ the \textit{symmetrization} of $\pi.$\\
\vspace{-3pt}

Every orthogonal complex representation  $\Pi$ of $G$ can be decomposed as 
\begin{equation}\label{decorth}\Pi\cong  \bigoplus_i \pi_i \oplus\bigoplus_j S(\varphi_j), \end{equation}
such that each $\pi_i$ is irreducible orthogonal and $\varphi_j$ are irreducible non-orthogonal representations of $G$. 

\begin{defn}
A complex representation $\pi$ of $G$ is said to be an \textit{orthogonally irreducible representation} (OIR), provided $\pi$ is orthogonal, and can not be decomposed into a direct sum of orthogonal representations.
\end{defn}

 An irreducible representation $\pi$ is \textit{orthogonally irreducible} if and only if $\pi$ is orthogonal. Moreover, for $\varphi$ irreducible and non-orthogonal, its symmetrization $S(\varphi)$ is an OIR.  
In fact, the decomposition in \eqref{decorth} establishes that all the OIRs of $G$ are either irreducible orthogonal $\pi$, or of the form $S(\varphi)$ with $\varphi$ irreducible and non-orthogonal. 

\subsection{Stiefel-Whitney Classes}

In this section we define Stiefel-Whitney classes for orthogonal complex representations of finite groups. This is equivalent to the definition in \cite{Benson} or \cite{GKT}, given for real representations.

Recall from  \cite{milnor} that a ($d$-dimensional real) vector bundle $\mc V$ over a paracompact space $B$ is assigned a sequence 
\beq
w_1(\mc V), \ldots, w_d(\mc V)
\eeq
of cohomology classes, with each $w_i(\mc V) \in H^*(B,\Z/2\Z)$.

For every finite group $G$ there is a classifying space $BG$ with a principal $G$-bundle $EG$, unique up to homotopy. From a real representation $(\rho,V)$ one can form the associated vector bundle $EG[V]$ over $BG$. Then one puts
\beq
w_i^\R(\rho)=w_i(EG[V]);
\eeq
see for instance \cite{Benson} or  \cite{GKT}. 
Moreover the singular cohomology $H^*(BG,\Z/2\Z)$ is isomorphic to the group cohomology $H^*(G,\Z/2\Z)$.  At this point we simply write $H^*(G)=H^*(G,\Z/2\Z)$.

Hence, to a real representation $\rho$, we can associate cohomology classes $w_i^\R(\rho) \in H^i(G)$, called Stiefel-Whitney classes (SWCs).
We prefer to work with an orthogonal complex representation $\pi$, which by Proposition  \ref{prop.one} comes from a unique real representation $\pi_0$. 
We thus define
$$w_i(\pi):=w_i^\rr(\pi_0),$$
for $0\leq i\leq \deg\pi$.

  A group homomorphism $\varphi: G_1 \to G_2$ induces a map 
  $
  \varphi^*: H^*(G_2) \to H^*(G_1).
  $
  The SWCs are \textit{functorial} in the following sense: If  $\pi$ is an orthogonal representation of $G_2$, then  
  \begin{equation} \label{funct.w}
  \varphi^*(w(\pi))=w(\pi \circ \varphi).
  \end{equation}
  
The SWCs are also \textit{multiplicative.} This means if $\pi_1$ and $\pi_2$ are both orthogonal, then \begin{equation}\label{p4}w(\pi_1\oplus \pi_2)=w(\pi_1)\cup w(\pi_2).\end{equation}
 The first few $w_i(\pi)$ have nice interpretations. Firstly, $w_0(\pi)=1 \in H^0(G)$ and the first SWC, applied to linear characters $G \to \{\pm 1\} \cong \Z/2\Z$, is the well-known isomorphism
  \beq
  w_1: \Hom(G, \pm1)  \overset{\cong}{\to} H^1(G).
  \eeq

More generally, if $\pi$ is an orthogonal representation, then $w_1(\pi)=w_1(\det \pi)$, where $\det \pi$ is simply the composition of $\pi$ with the determinant map.
When $\det \pi=1$, then $w_2(\pi)$ vanishes iff $\pi$ lifts to the corresponding spin group. (See \cite{Ganguly} for details.)
\bigskip

To complex representations $\pi$ of $G$, are associated cohomology classes $c_i(\pi)\in H^{2i}(G,\z)$, called \textit{Chern classes}. As with SWCs, $c_0(\pi)=1$ and the first Chern class, applied to linear characters, gives the well-known isomorphism
$$c_1:\Hom(G,S^1)\xrightarrow{\cong} H^2(G,\z).$$
More generally, $c_1(\pi)=c_1(\det\pi)$ for a complex representation $\pi$.

The SWC of a symmetrization $S(\pi)$ can be obtained from the Chern class of $\pi$ via the formula
\begin{equation}\label{ksc}w(S(\pi))=\kappa(c(\pi)).\end{equation}
Here $\kappa: H^*(G,\z)\to H^*(G,\z/2\z)$ is the \emph{coefficient homomorphism} of cohomology. 
(See \cite[Lemma 1]{GJgl2}, based on \cite[Problem 14-B]{milnor}.)
 
From this we make an observation for later use: 
\begin{lemma}\label{spi12}
If $\det \pi=1$, then $w_i(S(\pi))=0 \text{ for }i=1,2,3$.
\end{lemma}
\begin{proof}
Since Chern classes live in even degrees, all the odd SWCs of $S(\pi)$ vanish. Moreover $c_1(\pi)=0$ by the determinant hypothesis.
\end{proof}
 \subsection{Steenrod Squares}
For $n,i\geq 0$, there are operations on cohomology, called \textit{Steenrod Squares},
 $$\Sq^i:H^n(G)\rightarrow H^{n+i}(G).$$

These are additive homomorphisms and $\Sq^0$ is the identity. Steenrod operations are \textit{functorial}, meaning for a group homomorphism  $\varphi: G\to G'$, we have $$\varphi^*(\Sq^iy)=\Sq^i(\varphi^*(y)) \text{ for all }y\in H^n(G').$$
 
We now state the well-known \textit{Wu formula}:
\begin{prop}[\cite{may1999concise}, Chapter 23, Section 6]\label{Wu}
Let $\pi$ be an orthogonal representation of $G$. Then, the cohomology class $\Sq^i(w_j(\pi))$ can be expressed as a polynomial in $w_1(\pi),\hdots,w_{i+j}(\pi)$:
$$\Sq^i(w_j(\pi))=\sum_{t=0}^i{j+t-i-1\choose t}w_{i-t}(\pi)w_{j+t}(\pi).$$
\end{prop}
With $i=1$, $j=2$ for instance, we can express $w_3(\pi)$ in terms of $w_1,w_2$ as follows:
\begin{equation*}\label{w3}w_3(\pi)=w_1(\pi)\cup w_2(\pi)+\Sq^1(w_2(\pi)).\end{equation*}

\subsection{Elementary Abelian $2$-groups}\label{ea2}  
 Let $C_2=\{ \pm 1\},$ the cyclic group of order $2$. Here we review explicit descriptions of group cohomology for elementary abelian 2-groups $C_2^n$ and describe SWCs in terms of these descriptions.

  We write `sgn' for the order $2$ linear character of $C_2$. It is known \cite{KT} that 
\begin{equation}\label{c2}H^*(C_2)\cong\z/2\z[v],\end{equation}
where  $v=w_1(\sgn)$. Since this group is so simple, we can give our first example of expressing SWCs in terms of character values.  (Note that every representation of $C_2$ is orthogonal.) 

\begin{lemma}\label{c2swc} Let $\pi$ be a representation of $C_2$, and $b_\pi=\frac{1}{2} \left( \deg \pi-\chi_\pi(-1) \right)$. Then, the total SWC of $\pi$ is
\beq
w(\pi)=(1+v)^{b_\pi}.
\eeq
\end{lemma}
\begin{proof}
We can write
$$\pi=a1\oplus b(\sgn)$$
for non-negative integers $a,b$. From the multiplicativity of SWCs, we obtain
$$w(\pi)=(1+v)^b.$$
To express $b$ in character values, consider the equations: 
\begin{align*}
\chi_\pi(1)&=a+b\\
\chi_\pi(-1)&=a-b
\end{align*}
giving $b=\cfrac12(\chi_\pi(1)-\chi_\pi(-1))=b_\pi.$\\
\end{proof}

Let $C_2^n$ be the $n$-fold product of $C_2$, with projection maps $\pr_i: C_2^n\to C_2$ for $1 \leq i \leq n$. By K\"{u}nneth, we have 
$$H^*(C_2^n)=\z/2\z[v_1,\dots,v_n]$$
where we put  $v_i=w_1(\sgn\circ\pr_i)$. 
\subsection{SWCs for Virtual Representations}\label{vr}

A virtual representation of $G$ can be thought as a difference $\pi=\pi_1 \ominus \pi_2$, for representations $\pi_1,\pi_2$. When the $\pi_i$ are orthogonal, one may define the SWC of $\pi$ as 
\beq
w(\pi)=w(\pi_1) \cup w(\pi_2)^{-1},
\eeq
but we must ``complete'' the cohomology ring so that the inversion makes sense. 

More formally, let $G$ be a finite group, and let $\RO(G)$ be the free abelian group generated by the isomorphism classes of OIRs (following  \cite[Chapter II, Section 7]{BrokerDieck}).  Members of $\RO(G)$ are called \emph{virtual orthogonal representations} of $G$.

Write \begin{equation*}
H^\bullet(G)=\prod_i H^i(G),
\end{equation*}
for the complete mod $2$ cohomology ring of $G$, consisting of all formal infinite series $\alpha_0+\alpha_1+ \cdots$, with $\alpha_i \in H^i(G)$. (Please see \cite[page 44]{weibel2013k}.) Each $w(\pi)$ is invertible in this ring.

It is clear now that we may ``extend'' $w$ to a group homomorphism, taking values in the unit group of $H^\bullet(G)$, i.e.,
\begin{equation} \label{extend}
 w: \RO(G) \to H^\bullet(G)^\times.
 \end{equation}

Later in this paper we shall determine the image of $w$ for $G=\SL(2,q)$.

   \section{Quaternion Groups}\label{quatsec}
   
 Our proof of Theorem \ref{prop} involves subgroups of $\SL(2,q)$ isomorphic to certain quaternion groups.  In this section we review the group cohomology of the quaternion group of order $8$, and the so-called ``generalized quaternions" $Q_{2^n}$. This will play a vital role in our detection theorem.
   
  \subsection{Orthogonal Representations of $Q$} 
Let $Q$ be the Quaternion group of order 8. If we write $\mb H$ for the familiar real division algebra of quaternions, then $Q$ can be quickly described as the subgroup of $\mb H^\times$ generated by $i$ and $j$.

There are 4 one-dimensional representations of $Q$, which we denote by 1, $\chi_1,\chi_2,\chi_3=\chi_1\otimes\chi_2$. Each one is orthogonal.
The group $Q$ also possesses a unique 2-dimensional irreducible representation 
$$\rho:Q\rightarrow\SL(2,\cc)$$
defined by
\begin{equation*}\label{rhoij}\rho(i)=\begin{pmatrix}
i&&0\\
0&&-i
\end{pmatrix} \text{  ,  }\; \rho(j)=\begin{pmatrix}
0&&1\\
-1&&0
\end{pmatrix}.
\end{equation*}
 
 Of course, $\hh$ can be described as the algebra of complex matrices of the form 
$$\begin{pmatrix}
s+ti & u+vi \\
-u+vi& s-ti
\end{pmatrix} \text{ where } s,t,u,v\in \rr.$$  Since both $\rho(i)$, $\rho(j)$ are of this form,  $\rho$ is an injection of $Q$  into  $\hh^1$, the subgroup of norm $1$ quaternions. In particular, $\rho$ is symplectic.
Thus, $S(\rho)=\rho\oplus\rho$ is the only OIR of $Q$, which is not irreducible.

\subsection{The Cohomology ring $H^*(Q)$}\label{cohoQ} 
The SWCs of orthogonal representations of $Q$ are certain elements in $H^*(Q)$. Here we give a description of the cohomology ring $H^*(Q)$ with \cite[Chapter IV]{Milgram} as \enlargethispage{\baselineskip} our reference.

Recall that $w_1$ gives an isomorphism between $\Hom(Q,\pm1)=\{1,\chi_1,\chi_2,\chi_3\}$ and $H^1(Q)$. We define
\begin{align*}
x&:=w_1(\chi_1), \text{ and} \\
y&:=w_1(\chi_2).
\end{align*} 

The cohomology group $H^4(Q)$ is one-dimensional over $\z/2\z$, and we write `$e$' for its nonzero element.

\begin{prop}[\cite{Milgram}, Section V.1]\label{cohomQ}
 The cohomology ring of $Q$ is
$$H^*(Q)\cong \z/2\z[x,y,e]/{(xy+x^2+y^2,x^2y+xy^2)}.$$
\end{prop}

The first few cohomology groups of $Q$ are as follows:
\begin{align*}
H^0(Q)&=\{0,1\}\\
H^1(Q)&=\{0, x, y,x+y\}\\
H^2(Q)&=\{0, x^2,y^2,x^2+y^2\}\\
H^3(Q)&=\{0, x^2y\}\\
H^4(Q)&=\{0, e\}.
\end{align*}
Note that $x^3=y^3=0$. Moreover, one can obtain the higher cohomology groups via the cup product with $e$ due to Proposition \ref{cohomQ}.  This means
 for $i\geq0,$ the map
\begin{align*}
H^i(Q)&\to H^{i+4}(Q)\\
z&\mapsto z\cup e
\end{align*}
is an isomorphism of groups.

\subsection{SWCs of Representations of  $Q$}\label{SWquat}
We first compute the total SWCs of OIRs of $Q$.
For the linear orthogonal representations of $Q$, it is clear that 
$w(1)=1,$ $w(\chi_1)=1+x$, $w(\chi_2)=1+y$, and $ w(\chi_1\otimes\chi_2)=1+x+y$ with $x,y\in H^1(Q)$  from above.

\begin{lemma}\label{Srho}
The total SWC of $S(\rho)$ is $$w(S(\rho))=1+e,$$ where $e$ is the non-trivial element of $H^4(Q)$.
\end{lemma}

\begin{proof}
From Lemma \ref{spi12}, we have $$w_i(S(\rho))=0\text{ for }i=1,2,3.$$ This gives $w(S(\rho))=1+w_4(S(\rho))$.
We now prove that $w_4(S(\rho))=e\in H^4(Q).$\\

\noindent Let $Z$ be the center of $Q$, which is $\{ \pm 1\}$. Since
\begin{align*}
\rho(-1)= \begin{small}\begin{pmatrix}
-1& 0\\
0&-1
\end{pmatrix},\end{small}
\end{align*}
we have the restriction $\res_{Z}^Q\rho=\sgn\oplus\sgn,$
where $\sgn:Z\rightarrow \cc^\times$ is the linear character of order $2$.

From \eqref{c2}, we have $H^*(Z)\cong \z/2\z[v]$ where $v=w_1(\sgn)$. Then, the multiplicativity of SWCs gives
\begin{align*}
w(\res_{Z}^QS(\rho))=w(\sgn)^4=(1+v)^4=1+v^4,
\end{align*}
which implies $w_4(S(\rho))\neq0$. Therefore, $w_4(S(\rho))=e$, since this is the only nonzero element in $H^4(Q)$. 
\end{proof} 

We also see from the above proof that $e$ maps to $v^4$ under the restriction $H^4(Q)\to H^4(Z)$. From this, we deduce that the restriction maps in degrees divisible by $4$,
 \begin{equation} \label{qd}
\begin{split}
H^{4i}(Q)&\rightarrow H^{4i}(Z),\\
e^i&\mapsto v^{4i}
\end{split}
\end{equation}
are isomorphisms.

Let $ \pi$ be an orthogonal representation of $Q$. By \eqref{decorth}, we can write
$$\pi \cong m_01\oplus m_1\chi_1\oplus m_2\chi_2\oplus m_2 \chi_3\oplus m_4 S(\rho),$$
where $m_i$ are non-negative integers.
\begin{lemma}
Let $\pi$ be as above. Then,  for $0\leq4i\leq\deg\pi$, we have $$w_{4i}(\pi)=\binom{m_4}{i}e^i,$$ where $m_4=\frac18(\chi_{\pi}(1)-\chi_{\pi}(-1))$.
\end{lemma}
\begin{proof}
Let us consider $$S= \z/2\z[x,y]/{( xy+x^2+y^2,x^2y+xy^2)},$$ which is a $6$-dimensional subalgebra of $H^*(Q)$. By the multiplicativity of SWCs, we have
\begin{align*}
w(\pi)&=w(\chi_1)^{m_1}\cup w(\chi_2)^{m_2}\cup w(\chi_3)^{m_3}\cup w(S(\rho))^{m_4}\\
&=\underbrace{(1+x)^{m_1}(1+y)^{m_2}(1+x+y)^{m_3}}_{\quad\in \:S}(1+e)^{m_4}.
\end{align*}
Therefore, we can write $(1+x)^{m_1}(1+y)^{m_2}(1+x+y)^{m_3}$ as a polynomial of the form $P(x,y)=1+Ax+By+Cx^2+Dy^2+Exy^2$ with coefficients in $\z/2\z,$ and  \begin{align*}
w(\pi)&=P(x,y)(1+e)^{m_4}\\
&=P(x,y)\sum_{i=0}^{m_4}\binom{m_4}{i}e^i
\end{align*}
where $e^i\in H^{4i}(Q)$ for each $i$. Comparing degrees gives

$$w_{4i}(\pi)=\binom{m_4}{i}e^i.$$
To express $m_4$ in terms of character values, we simply evaluate $\chi_{\pi}$ at 1 and $-1$. Since $-1\in $Ker$(\chi_i)$ for $i=1,2,3$ and $\chi_\rho(1)=2$, $\chi_\rho(-1)=-2$, we have the equations:
\begin{align*}
\chi_\pi(1)&=m_0+m_1+m_2+m_3+4m_4,\\
\chi_\pi(-1)&=m_0+m_1+m_2+m_3-4m_4,
\end{align*}
and so $m_4=\frac18(\chi_{\pi}(1)-\chi_{\pi}(-1))$.
\end{proof}
\subsection{Generalized Quaternions}
The construction of the quaternion group $Q$ generalizes to give a family of non-abelian groups which have the presentation
$$Q_{2^n}:=\langle a,b\mid   a^{2^{n-2}}=b^2, b^4=1,bab^{-1}=a^{-1} \rangle\: ;\: n\geq3.$$
These groups are called \textit{generalized quaternion groups} and have order $2^n$. 
Again, there are 4 linear (orthogonal) representations of $Q_{2^n}$, say $1,\psi_1,\psi_2,\psi_3.$ These can be defined on the generators of $Q_{2^n}$ as:
\begin{equation}
\label{psi}
 \begin{alignedat}{4}
\psi_1(a)&=-1\hspace{3mm}&,\hspace{3mm}& \psi_1(b) &=\;\;\;1\\
\psi_2(a)&=-1\hspace{3mm}&,\hspace{3mm}& \psi_2(b)&=-1\\
\psi_3(a)&=1\hspace{3mm}&,\hspace{3mm}& \psi_3(b)\:&=-1.
\end{alignedat}
\end{equation}
Let $\zeta=e^{2\pi i/2^{n-1}}$. There is an irreducible 2-dimensional representation $\varrho$ of $Q_{2^n}$   defined by:
$$\varrho(a)=\begin{small}\begin{pmatrix}
\zeta&&0\\
0&&\bar{\zeta}
\end{pmatrix}\end{small}\quad,\quad\varrho(b)= \begin{small}\begin{pmatrix}
0&&1\\
-1&&0
\end{pmatrix}.\end{small}$$
As earlier, $\varrho$ maps  $Q_{2^n}$ into $\hh^1$, so it is symplectic. 

Let $n>3$. We use the isomorphism 
\begin{align*}
w_1: \Hom(Q_{2^n},\pm 1) &\xrightarrow{\cong} H^1(Q_{2^n})\\
\{1,\psi_1,\psi_2,\psi_3\}&\leftrightarrow \{0,X,Y,X+Y\}.
\end{align*}
to define  $X=w_1(\psi_1)$ and $Y=w_1(\psi_2)$. 

A description of the mod $2$ cohomology ring of generalized quaternions $Q_{2^n}$ for $n>3$ in  \cite[Chapter IV, Lemma 2.11]{Milgram} is as follows:
\begin{equation}\label{CQn}
H^*(Q_{2^n})\cong \z/2\z[X,Y,E]/(XY,X^3+Y^3),
\end{equation}
where $E\in H^4(Q_{2^n})$.
\begin{lemma}\label{crgq}
We have $E=w_4(S(\varrho))$ in $H^*(Q_{2^n}).$
\end{lemma}
\begin{proof}
This can be proved precisely as in Lemma \ref{Srho}; again $w(S(\varrho))=1+E$.
\end{proof}

\begin{remark}
Let $n>3$. Let  $i:Q\rightarrow Q_{2^n}$ be any group homomorphism. Then, 
\begin{align*}i^*:H^3(Q_{2^n})&\rightarrow H^3(Q)
\end{align*}
is  the zero map. This can be seen as follows. 
\end{remark}
\noindent From \eqref{CQn}, we can deduce $H^3(Q_{2^n})=\{0,X^3\}$ and since $i^*$ is a ring homomorphism, we have the image $i^*(X)\in H^1(Q) =\{0,x,y,x+y\}$. But $x^3=y^3=(x+y)^3=0$  in $H^*(Q)$, which implies
\begin{align*}i^*(X^3)=(i^*(X))^3=0.{\qed}\end{align*}

Consider the subgroup $Q^{(1)}=\langle a^{2^{n-3}},b\rangle\leqslant Q_{2^n}$. This subgroup is isomorphic to $Q$ by $i \leftrightarrow a^{2^{n-3}}$ and $j \leftrightarrow b$. With this identification, it is easy to see that
\beq
\res_Q^{Q_{2^n}} \varrho = \rho.
\eeq
Write $\iota_Q$ for this inclusion of $Q$ into $Q_{2^n}$. Now
\beq
\begin{split}
\iota_Q^*(E) &=\iota_Q^*(w_4(\varrho)) \\
		&=w_4(\res_Q^{Q_{2^n}} \varrho) \\
		&=w_4(\rho) \\
		&=e. \\
		\end{split}
		\eeq
This means the homomorphism $\iota^*_Q:H^*(Q_{2^n})\to H^*(Q)$, is an isomorphism in degree $4$. Therefore:

\begin{prop}\label{gqdp}With notation as above, we have isomorphisms 
\begin{equation}\label{gqd}
\begin{split}
{\iota_Q^*}:H^{4i}(Q_{2^n})&\rightarrow H^{4i}(Q)\\
E^i&\mapsto e^i
\end{split}
\end{equation}
\enlargethispage{\baselineskip}
for all $i\geq 0$.
\end{prop}

\section{$\protect\SL(2,q)$ with $q$ odd }\label{odd}
 Let $G=\SL(2,q)$ throughout this section, with $q$ odd.  Write `$\mathbbm{1}$' for the identity matrix; then the center $Z$ of $G$ is  $\{\pm\mathbbm{1}\}$. 
 We begin with a description of the mod $2$ cohomology of $G$.
\begin{prop}[\cite{fiedorowicz2006homology}, Chapter VI, Sec. 5]\label{crsl2}
When $q$ is odd, the mod $2$ cohomology ring of $\SL(2,q)$ is
$$H^*(\SL(2,q))\cong \z/2\z[\mathfrak{e}]\otimes \z/2\z[\mathfrak{b}]/\langle \mathfrak{b}^2\rangle$$
\enlargethispage{\baselineskip}
with $\deg(\mathfrak{b})=3$, $\deg(\mathfrak{e})=4$.
\end{prop}
\subsection{Detection}\label{detect}
The proof of Theorem \ref{prop} uses the following well-known detection:
\begin{lemma}[\cite{Milgram}, Cor. 5.2, Ch. II]\label{syl}
Let $H$ be a finite group. A Sylow $2$-subgroup of $H$ detects its mod $2$ cohomology.  \end{lemma}

According to \cite[Chapter 2, Theorem 8.3]{gorenstein1980finite}, a Sylow $2$-subgroup of $G$ is generalized quaternion. We will therefore regard $Q_{2^n}$ as a subgroup of $G$, where of course $n$ is the highest power of $2$ 
dividing $|G|$.    We now prove our main detection theorem.

\begin{proof}[Proof of Theorem \ref{prop}]
We deduce from Proposition \ref{crsl2} that
$$H^n(G)\cong\begin{cases}
\z/2\z, & n\equiv 0,3\; (\text{mod } 4)\\
0, &\text{ otherwise}.
\end{cases}$$
Let $\pi$ be an orthogonal representation of $G$. Let $n=4k+3$ for some $k$. With $i=1$, $j=n-1$ such that $i+j=n$, we now apply the Wu formula from Proposition \ref{Wu}:
\begin{align*}
\Sq^1(w_{n-1}(\pi))&=w_1(\pi)\cup w_{n-1}(\pi)+{n-2\choose 1}w_0(\pi)\cup w_n(\pi)
\end{align*}
Here, $\Sq^1(w_{n-1}(\pi))=0$ because $w_{n-1}(\pi)\in H^{4k+2}(G)$ is zero and $\Sq^1$ is a homomorphism. Also, $w_1(\pi)=0$ and $(n-2)$ is odd which implies
$$w_n(\pi)=0 \text{ when } n\equiv 3\text{ (mod }4).$$
Therefore, the non-zero SWCs of representations of $G$ only occur  in degrees divisible by $4$, which means
$$H^*_{\SW}(G)\subseteq \z/2\z[\mathfrak{e}].$$
Since $Q_{2^n}$ is a $2$-Sylow subgroup of $G$, it detects the mod $2$ cohomology of $G$ by Lemma \ref{syl}. 
Consider the Quaternion subgroup $Q \cong Q^{(1)}\leqslant Q_{2^n}$. Note that the center of $Q$ is $Z$. By the isomorphisms \eqref{qd} and \eqref{gqd}, we obtain the following sequence of inclusions for each $i\geq 0$:
\begin{alignat*}{3}
H^{4i}(G)&\hookrightarrow H^{4i}(Q_{2^n})&&\hookrightarrow H^{4i}(Q)&&\hookrightarrow H^{4i}(Z)\\
\mathfrak{e}^i\quad&\mapsto\quad E^i&&\mapsto\quad e^i\quad&&\mapsto\quad v^{4i}.
\end{alignat*}
Therefore, the subalgebra $\z/2\z[\mathfrak{e}]$, containing the SWCs of $G$, injects into $H^*(Z)$, completing the proof.

\end{proof}

\subsection{Formulas for SWCs}\label{thms}
We now prove Theorem \ref{thm2.section} which gives an explicit formula for the total SWCs of orthogonal representations of $G$ in terms of character values.

\begin{proof}[Proof of Theorem \ref{thm2.section}]
  
To find $w(\pi)$, it is enough to work with $\res^G_Z \pi$  due to the detection Theorem \ref{prop}. This restriction has been done in two steps. We first restrict $\pi$ to the quaternion group $Q\leqslant G$, and then further from $Q$ to $Z$.
 
 With notations from Section \ref{quatsec}, it looks like
$$\res^G_Q\pi\cong m_01\oplus m_1\chi_1\oplus m_2\chi_2\oplus m_3\chi_3\oplus m_4 (S(\rho)) $$
where $m_i$ are non-negative integers. It is easy to see that 
\begin{equation*}\label{resQZ}
\begin{split}
\res^Q_Z \chi_i&\cong 1 \text{ for }i=1,2,3, \text{ and }\\
\res^Q_Z\rho&\cong \sgn\oplus \sgn.
\end{split}
\end{equation*}
Therefore, the further restriction of $\res^G_Q\pi$ to the center $Z$ will be $$\res^G_Z\pi=\res^Q_Z\res^G_Q\pi\cong (m_0+m_1+m_2+m_3)  1\oplus 4m_4 (\sgn).$$
 From Lemma \ref{c2swc}, we then obtain
$$w(\res^G_Z \pi)=(1+v)^{4m_4}=(1+v^4)^{m_4},$$
with $4m_4=\frac12(\chi_\pi(\mathbbm1)-\chi_\pi(-\mathbbm1)).$ From the proof of Theorem \ref{prop}, we have $i^*_Z(\mathfrak{e})=v^4$ where $i_Z:Z\rightarrow G$ is the inclusion. Therefore,
$$w(\pi)=(1+\mathfrak{e})^{r_\pi},$$
where $r_\pi=m_4$ is the multiplicity of $S(\rho)$ in $\res^G_Q\pi$.

\end{proof}

\begin{ex} We have
\beq
w(\reg(G))=(1+\mf e)^{\frac{1}{8}|G|}.
\eeq
\end{ex}

\subsection{Corollaries}

We can go deeper by applying the following result due to R. Gow: 
\begin{thm}[\cite{gow1985real}, Theorem 1]\label{cen1}
 Let $\pi$ be an irreducible self-dual representation of $G$ with central character $\omega_\pi.$ Then, the Frobenius-Schur Indicator $\varepsilon(\pi)$ equals $\omega_\pi(-\mathbbm1)$.
\end{thm}
We simply call this equality \emph{Gow's formula}.\\

 This leads to the following:
\begin{c.intro}\label{thm1.section}  Let $\pi$ be an irreducible orthogonal representation of $G$. Then, its  total SWC  $w(\pi)=1$.
\end{c.intro}
\begin{proof} For $\pi$ irreducible orthogonal,  the sign $\varepsilon(\pi)=1$ by definition.
Therefore 
\begin{align*}
r_\pi&=\frac18(\chi_\pi(\mathbbm1)-\chi_\pi(-\mathbbm1))\\
&=\frac{\chi_\pi(\mathbbm1)}{8}(1-\omega_\pi(-\mathbbm1))\\
&=0.
\end{align*} \end{proof}
\vspace{-3pt}
Let $\pi$ be an irreducible, non-orthogonal representation of $G$. Then, for the \enlargethispage{\baselineskip} representation $S(\pi)=\pi\oplus\pi^\vee$, we have
\begin{align*}
r_{S(\pi)}&=\frac18\big(\chi_{S(\pi)}(\mathbbm1)-\chi_{S(\pi)}(-\mathbbm1)\big)\\
&=\frac14\big(\chi_\pi(\mathbbm1)-\chi_\pi(-\mathbbm1)\big)\\
&=\frac{\chi_\pi(\mathbbm1)}{4}\big(1-\omega_\pi(-\mathbbm1)\big).
\end{align*}
Furthermore, for symplectic $\pi$, it turns out to be
\begin{equation}\label{spi}
r_{S(\pi)}=\frac{\chi_\pi(\mathbbm1)}{2}=\frac{\deg(\pi)}{2}
\end{equation}
 due to Gow's formula.
\begin{c.intro}\label{c1}
  Let $G=\SL(2,q)$ with $q$ odd. Then the image of $w$ in \eqref{extend} is 
$$
  \{ (1+\mathfrak{e})^n \mid n \in \Z\}.$$

\end{c.intro}
\begin{proof}
For $q=3$, there exists a unique irreducible symplectic representation $\pi_0$ of $G$ with degree $2$. From Equation \eqref{spi}, we obtain $r_{S(\pi_0)}=1$, and so $w(S(\pi_0))=1+\fe.$
\vspace{0.5mm}

 For $q>3$, there exist irreducible symplectic representations $\pi_1$ and $\pi_2$ of $G$ of degrees $q+1$ and $q-1$ respectively. (See Remark \ref{psc} for details.)
From Equation \eqref{spi}, we have
$$r_{S(\pi_1)}=\frac{q+1}{2}\;, \; r_{S(\pi_2)}=\frac{q-1}{2} $$
which are co-prime. So by B\'{e}zout's Identity, there exist integers $a,b$ such that $a\big(\frac{q+1}{2}\big)+b\big(\frac{q-1}{2}\big)=1$. \\

\noindent Therefore, there is a virtual representation $\pi\in\RO(G)$ with $r_\pi=1$ such that
$$w(\pi)=1+\mathfrak{e}.$$
Hence, the result follows from the multiplicativity of SWCs.
\end{proof}
\begin{remark}\label{psc}
  For completeness, we give explicit descriptions of $\pi_1$ and $\pi_2$. This entails recalling the construction of the principal series and irreducible cuspidal representations of $G$. 
 The reader is referred to \cite{Fulton.Harris} and \cite{Bushnell.Henniart} for more details.  
\vspace{2mm}

 Let $\widetilde T$ be the diagonal subgroup, isomorphic to $\f_q^\times\times\f_q^\times$, contained in the standard Borel subgroup $\widetilde B$ of upper-triangular matrices of $\widetilde G=\GL(2,q)$. 
When $\alpha,\beta$ are linear characters of $\f_q^\times$, we will write $\alpha \boxtimes \beta$ for the corresponding linear character of $\widetilde T$.
 Choose $\as:\f_q^\times\to\cc^\times$ satisfying
\begin{align*}
\as(-1)&=-1\\
\as^2&\neq1.
 \end{align*}
One can inflate  $\as\boxtimes1$ from $\widetilde T$ to $\widetilde B$ and then consider the usual complex parabolic induction $\pi_\as:=\Ind_{\widetilde B}^{\widetilde G}(\as\boxtimes1)$. This is an irreducible principal series representation of $\widetilde G$ of degree $q + 1$.
\vspace{1mm}

 We take $\pi_1=\res^{\widetilde G}_G\pi_\as$. This restriction is self-dual and irreducible. By Gow's formula, $\pi_1$ is symplectic.

 Let $\widetilde T_e$ be an elliptic torus of $\widetilde G$, thus isomorphic to $\f_{q^2}^\times$. Let $\widetilde{Z}$ be the center and $N$ be the upper unitriangular subgroup of $\widetilde G$. Choose a linear character $\chi$ of $\widetilde T_e$ such that
\begin{align*}
\tht^q\neq \tht,&\;\; \tht^2\neq 1\\
\tht(-1)&=-1.
\end{align*} 
We fix a nontrivial character $\varphi$ of $N$, and define a linear character of $\widetilde{Z} N$ as $\chi_\varphi(zn)=\chi(z)\varphi(n)$. Set
\[\pi_\chi=\Ind_{\widetilde{Z}N}^{\widetilde G}\chi_\varphi-\Ind_{\widetilde T_e}^{\widetilde G}\chi.\]

\noindent This is an irreducible, cuspidal representation of $\widetilde G$ of dimension $q-1$. 

When restricted to $G$, $\pi_\chi$ remains irreducible and is called an irreducible cuspidal representation of $G$. (A representation of $G$ is called \emph{cuspidal}, provided each of its irreducible constituents is irreducible cuspidal.)
\vspace{1mm}

 Define $\pi_2=\res^{\widetilde G}_G\pi_\chi$. Again one sees that $\pi_2$ is symplectic by Gow's formula.
\end{remark}
It is already known from Theorem \ref{prop} that $H^*_{\SW}(G)\subseteq \z/2\z[\mathfrak{e}]$. Here we make a stronger statement:
\begin{c.intro}\label{eta}
Let $G=\SL(2,q)$ with $q$ odd. Then, $$H^*_{\SW}(G)=\z/2\z[\mathfrak{e}].$$
\end{c.intro}
\begin{proof}
For equality, we construct an orthogonal representation $\eta$ of $G$ such that $w_4(\eta)=\mathfrak{e}$ as follows. 

For $q>3$, consider the irreducible symplectic representations $\pi_1$, $\pi_2$ of $G$ from Corollary \ref{c1}. Let $\eta=S(\pi_1)\oplus S(\pi_2)$. It can easily seen that $r_{\pi\oplus\pi'}=r_\pi+r_\pi'$  for any orthogonal $\pi,\pi'$.  Therefore, we have
$$r_\eta=r_{S(\pi_1)\oplus S(\pi_2)}=r_{S(\pi_1)}+r_{S(\pi_2)}=\frac{q+1}2+\frac{q-1}2=q.$$
Since $q$ is odd, we have $$w(\eta)=(1+\mathfrak{e})^{q}=1+\mathfrak{e}+\cdots. $$
Therefore, $w_4(\eta)=\mathfrak{e}.$

Moreover, we already have $w_4(S(\pi_0))=\fe$ for $q=3$ which completes the proof.
\end{proof}

Next, we give a nonvanishing condition for the top SWC.

\begin{c.intro}\label{ecodd}
 Let $G=\SL(2,q)$ with $q$ odd. Let $\pi$ be an orthogonal representation of $G$. Then, the following are equivalent:
\begin{roster} 
\item $w_{\Top}(\pi)\neq 0$.
\item $\pi=S(\varphi),$ where every irreducible constituent $\varphi_i$ of $\varphi$ has central character satisfying $\omega_{\varphi_i}(-\mathbbm1)=-1$.
\item $\pi(-\mathbbm1)=-1$, meaning $\pi(-\mathbbm 1)$ acts by the scalar $-1$. 
\end{roster}
\end{c.intro}

\begin{proof} 
 We first prove (i) and (ii) are equivalent. Assume $\pi=S(\varphi)=\bigoplus\limits_{i=1}^r S(\varphi_i)$ such that $\omega_{\varphi_i}(-\mathbbm1)=-1$ for all $1\leq i\leq r$.  
\vspace{2mm}

\noindent Clearly $\chi_\pi(\mathbbm1)=2\chi_\varphi(\mathbbm1)$, and since $\chi_{\varphi_i}(-\mathbbm1)=\omega_{\varphi_i}(-\mathbbm1)\chi_{\varphi_i}(\mathbbm1)=-\chi_{\varphi_i}(\mathbbm1)$, we have  
\begin{align*}
\chi_\pi(-\mathbbm1)=2\sum_{i=1}^r\chi_{\varphi_i}(-\mathbbm1)=-2\chi_\varphi(\mathbbm1).
\end{align*}
 This gives
$$
r_\pi=\frac12\chi_\varphi(\mathbbm1)=\frac14\chi_\pi(\mathbbm1).
$$
Therefore, we have $w(\pi)=(1+\mathfrak{e})^{\frac14\deg\pi}$, implying 
$$w_{\Top}(\pi)=\mathfrak{e}^{\frac14\deg\pi}\neq 0.$$
For the converse, suppose $w_{\Top}(\pi)\neq 0$. Being orthogonal, we can write $\pi$ as
$$\pi=\bigoplus_{i}\rho_i\oplus\bigoplus_j S(\phi_j)\oplus \bigoplus_k S(\varphi_k)$$
such that each $\rho_i$ is irreducible orthogonal, whereas $\phi_j,\varphi_k$ are irreducible non-orthogonal with $\omega_{\phi_j}(-\mathbbm1)=1$ for each $j$, and $\omega_{\varphi_k}(-\mathbbm1)=-1$ for each  $k$.\\
From Theorem \ref{thm2.section} and Corollary \ref{thm1.section}, we obtain
\begin{align*}
w(\pi)&=\prod_k(1+\mathfrak{e})^{\frac{\deg\varphi_k}{2}}\\
&=(1+\mathfrak{e})^{\frac{\deg\varphi}{2}}, \text{ where } \varphi=\bigoplus_k\varphi_k.
\end{align*}
Now, the condition $w_{\Top}(\pi)\neq 0$ implies $\deg\pi=4\cdot\frac{\deg\varphi}{2}$. Therefore,
\begin{align*}
\sum\limits_{i}\deg\rho_i+2\sum\limits_j \deg\phi_j +2\sum\limits_k \deg\varphi_k&=2\deg\varphi\\
\sum\limits_{i}\deg\rho_i+2\sum\limits_j \deg\phi_j&=0,
\end{align*}
which means $\rho_i$ and $S(\phi_j)$ don't appear in $\pi$. Hence $\pi=\bigoplus\limits_i S(\varphi_i)$, with each $\varphi_i$ irreducible and $\omega_{\varphi_i}(-\mathbbm1)=-1$. 

\vspace{2mm}

Next, suppose that $\pi_1, \ldots, \pi_r$ are representations of $G$, and $\pi$ is their direct sum. Then $\pi$ satisfies condition (iii) iff each $\pi_i$  satisfies condition (iii). (Because a block matrix equals the negative of the identity matrix iff each block is the negative of the identity matrix.) In particular, this shows that (ii) implies (iii). By Gow's formula, no irreducible orthogonal representation satisfies condition (iii), and so by decomposing $\pi$ we see that (iii) implies (ii).

\end{proof}

For $\pi$ orthogonal, let $k_0$ be the least $k>0$ such that $w_k(\pi)\neq 0$. Then $w_{k_0}(\pi)$ is known as the \textit{obstruction class} of $\pi$, following \cite{GJgl2}. 
\begin{c.intro}\label{obs}
 Let $\pi$ be an orthogonal representation of $G$. Put $t=\ord_2(r_\pi)$. Then the obstruction class of $\pi$ is $w_{2^{t+2}}(\pi)=\mf e^{2^t}$.
\end{c.intro}
\begin{proof}
By Theorem \ref{thm2.section},
$$w_{k}(\pi)=\begin{cases}{r_\pi\choose i}\mathfrak{e}^i ,\;&\;0\leq k=4i\leq \deg \pi\\
0,\;&\text{ otherwise}.
\end{cases}$$
Therefore the problem amounts to finding the smallest positive $k=4i$, so that the binomial coefficient  ${r_\pi\choose i}$ is odd. This is an elementary application of Lucas's Theorem  \cite{fine1947binomial}; proceeding as in 
\cite[Proposition 3]{GJgl2}, we find that the smallest such $k$ is $k=2^{t+2}$, giving the corollary.

\end{proof}

\subsection{Relation with $\GL(2,q)$}\label{GJ.sec}

Ganguly and Joshi compute SWCs for orthogonal representations  $\pi$ of $\widetilde G=\GL(2,q)$ in \cite{GJgl2}, with $q$ odd.  Let us review their result here, and see that it is harmonious with ours. 

Their SWCs are identified with members of $H^*(\widetilde T)$ for the bicyclic group $\widetilde T$, isomorphic to $C_{q-1} \times C_{q-1}$. 
When $q\equiv 3\:(\Mod\: 4)$, the cohomology is generated by elements $v_1,v_2$ of degree $1$, by K\"{u}nneth. 
When $q\equiv 1\:(\Mod\: 4)$, it is generated by elements $s_1,s_2,t_1,t_2$, where $\deg s_i=1$  and $\deg t_i=2$. Put $\delta=0$ if $\det \pi=1$ and $\delta=1$ otherwise.

In these terms, the 
 total SWC of $\pi$ is
\begin{equation}
\label{GJ}
w(\pi)=\begin{cases}
(1+\delta(s_1+s_2))(1+t_1+t_2+t_1t_2)^{r_\pi}(1+t_1+t_2)^{s_\pi/2},& q\equiv 1\;(\Mod\: 4)\\[1mm]
(1+v_1^2+v_2^2+v_1^2v_2^2)^{r_\pi}(1+v_1+v_2)^{s_\pi},& q\equiv 3\;(\Mod\: 4),
\end{cases}
\end{equation}
for a certain integer $s_\pi$, and   $r_\pi =\dfrac{1}{8}(\chi_\pi(\mathbbm1)-\chi_\pi(-\mathbbm1))$ as above.

Our two papers are  compatible via the following diagram, formed from the restriction maps on cohomology:
\[
\begin{tikzcd}
H^*(\widetilde G)\arrow[r]\arrow[d]&H^*(G)\arrow[d]\\
 H^*(\widetilde T)\arrow[r]&H^*(Z)
\end{tikzcd}
\]
 
 The RHS of \eqref{GJ} lies in $H^*(\widetilde T)$. The left map is injective by Quillen's work, and the right map is injective on SWCs by Theorem \ref{prop}. The lower map $H^*(\widetilde T) \to H^*(Z)$ maps $v_i$ to $v$ for $q\equiv 3\:(\Mod\: 4)$, and when $q\equiv 1\:(\Mod\: 4)$, it maps $s_i$ to $0$ and $t_i$ to $v^2$.
 This gives:
 \beq
 w(\pi|_Z)= (1+v^4)^{r_\pi}
\eeq
in all cases. Thus, for orthogonal representations of $G$ which extend to orthogonal representations of $\widetilde G$, our Theorem \ref{thm2.section} agrees with  the main result of \cite{GJgl2}.

\begin{remark} 
Not every orthogonal representation of $G$ extends to $\widetilde G$.  For instance when $q \equiv 1 \:(\Mod\: 4)$, the representation $\pi=\Ind_{\widetilde B}^{\widetilde G} (\sgn \boxtimes 1)$ (with $\widetilde B$ the upper triangular subgroup of $\widetilde G$), when restricted to $G$ is the sum of two irreducible orthogonal representations, neither of which extend to $\widetilde G$.
\end{remark}

\section{$\SL(2,q)$ with $q$ even}\label{even}
Let $G=\SL(2,q)$ with $q=2^r$ throughout this section. Recall that we write $N$ for the subgroup of upper unitriangular matrices. It is well-known that this is a Sylow $2$-subgroup of $G$, and therefore $N$ detects the mod $2$ cohomology of $G$ by Lemma \ref{syl}. All representations of $G$ are orthogonal by the main result in \cite{vinroot}. 

\subsection{Formulas for SWCs}
 
  The subgroup $N$ is an elementary abelian $2$-group of rank $r$, so as in Section \ref{ea2}, the mod $2$ cohomology ring of $N\cong C_2^r$ may be written as
$$H^*(N)\cong H^*(C^r_2)\cong \z/2\z[v_1,v_2,\hdots,v_r].$$

Let $T$ be the subgroup of diagonal matrices in $G$. That is
$$T=\Bigg\{\begin{small}\begin{pmatrix}
a&&0\\
0&& a^{-1}
\end{pmatrix}\end{small}: a\in \f_q^\times\Bigg\}.$$

The conjugation action of $T$ on $N$ is equivalent to the action of $\f_q^\times$ on $\f_q$ through multiplication by squares. Since $\f_q^\times$ has odd order, this action is transitive on $\f_q -\{0\}$. Therefore:
\begin{lemma}
$T$ acts transitively on the non-trivial linear characters of $N$. 
\end{lemma}

Let $\widehat N=\Hom(N,\cc^\times)$, the character group of $N$. From the above lemma, the $T$-orbits of $\widehat{N}$ are: $$\{1\}\;,\;\{\chi : \chi \neq 1\}.$$
Let us define
\beq
\reg(N)'=\reg(N) \ominus 1= \bigoplus_{\chi\neq 1}\chi.
\eeq

 Let $\pi$ be a  representation of $G$. 
Since $\res^G_N\pi$ is $T$-invariant, it is of the form
\begin{equation}\label{deg}
\begin{split}
\res^G_N\pi&\cong \ell_\pi1\oplus m_\pi \reg(N)',
\end{split}
\end{equation}
where $\ell_\pi,m_\pi$ are non-negative integers.
\begin{lemma}\label{mpi}
Let $\pi$ be a non-trivial irreducible representation of $G$. Then, $m_\pi=1$.
\end{lemma}
\begin{proof}
 For  $r\geq 2$, it is known that $G$ has no non-trivial normal subgroups. For non-trivial $\pi$, we must have $m_\pi\geq 1$ because ker$(\pi)=\mathbbm1$ and $N\not\leqslant\ker(\pi)$. Now from \eqref{deg}, we have
$$\deg\pi=\ell_\pi+m_\pi(q-1).$$
If $m_\pi>1$, then the sum $\ell_\pi+m_\pi(q-1)$ would be greater than the highest possible degree for an irreducible representation of $G$, which is $(q+1)$. So actually $m_\pi=1$.

For $r=1$, we have $\SL(2,2)\cong S_3$ with only two non-trivial irreducible representations. That is $\pi$ is either the `sgn' representation or the 2-dimensional standard representation of $S_3.$ Here we can see from direct calculations that $m_\pi=1$.
\vspace{2mm}

\end{proof}

It follows that $m_\pi=\dim_\C(V/V^G)$, where $V$ denotes the representation space of $\pi$. It is also the  number of non-trivial irreducible constituents of $\pi$. We can get a character formula for $m_\pi$ by taking character values at $n_0=\begin{pmatrix}
1&&1\\
0&&1
\end{pmatrix}\in N.$ In fact,
\begin{align*}
\chi_\pi(\mathbbm1)&=\ell_\pi+(q-1)m_\pi \text{ and }\\
\chi_\pi(n_0)&=\ell_\pi-m_\pi,
\end{align*}
so that $m_\pi=\frac1q(\chi_\pi(\mathbbm1)-\chi_\pi(n_0))$.

We are now ready to describe the SWCs of representations of $G$. Since $N$ is a detecting subgroup, we may and will identify $w(\pi)$ with its image in $H^*(N)$.

\begin{thm}\label{thm.even}
Let $q=2^r.$ Let $\pi$ be a representation of $\SL(2,q)$. Then, the total SWC of $\pi$ is
$$w(\pi)=\big(\prod_{v\in H^1(N)}(1+v)\big)^{m_\pi}.$$
 \end{thm}
\begin{proof}
The restriction $\res^G_N\pi$, from \eqref{deg}, is of the form
 $$\res^G_N\pi\cong (\ell_\pi-m_\pi)1\oplus m_\pi\reg(N).$$
Since $\reg(N)=\sum\limits_{\chi\in\widehat{N}}\chi$, we use the isomorphism between $\widehat{N}$ and $H^1(N)$ and the multiplicativity of SWCs to obtain 
\begin{align*}
w(\res^G_N\pi)&=w(\reg(N))^{m_\pi}\\
&=\big(\prod_{v\in H^1(N)}(1+v)\big)^{m_\pi}.
\end{align*}

\end{proof}
\begin{remark}
The formula in Theorem \ref{thm.even} above is also valid for $\widetilde G=\GL(2,q)$ with $q$ even, since $N$ is also a Sylow $2$-subgroup of $\widetilde G$. Thus this result in some sense completes \cite{GJgl2}.
\end{remark}

\subsection{Dickson Invariants}
The expansion of the product above is well-known. We have
\begin{equation} \label{Dickson}
\prod_{v\in H^1(N)}(1+v)=1+\sum_{i=1}^{r}d_{i}  \in H^*(N),\\
\end{equation}
where $d_{i}$ are \emph{Dickson invariants} in the algebra $H^*(N)$. 
 \\

We digress for a moment to recount the theory of Dickson invariants for mod $2$ spaces. Let $E$ be an $\f_2$-vector space. The ring of polynomials from $E$ to $\f_2$ can be identified with the symmetric algebra $S[E^\vee]$ on the dual space $E^\vee$. The linear group $\GL(E)$ acts on $S[E^\vee]$ via the contragradient map. It is natural to look for the invariants under this action.
\begin{thm}[\cite{wilkerson1983primer}]\label{wilk}
Suppose $\dim(E)=r$. Then the ring of invariants $S[E^\vee]^{\GL(E)}$ is a polynomial algebra  generated by elements $d_{i}$ called Dickson invariants, for $1\leq i\leq r$. The polynomial $d_{i}$ has degree $2^r-2^{r-i}$. Moreover,
$$\prod_{v\in E^\vee}(1+v)=1+\sum_{i=1}^{r}{d_{i}}\in S[E^\vee].$$
\end{thm}
These invariants are described explicitly in terms of certain determinants. (See \cite{wilkerson1983primer}, or \cite[Chapter III]{Milgram} for instance.)\\

Let us illustrate with some examples. If $E=\f_2$, then $S[E^\vee]$ is a polynomial algebra with only one generator $v$, that is $S[E^\vee]\cong \z/2\z[v]$. Here $d_{1}=v$ giving the ring of invariants $$S[E^\vee]^{\GL(V)}\cong\z/2\z[v]^{\GL(1,2)}=\z/2\z[v].$$
Also Theorem \ref{thm.even} for $G=\SL(2,2)$ and $\pi\neq 1$ irreducible says that 
$$w(\pi)=1+v=1+d_1.$$

Next, suppose $\dim E=2$. Then $S[E^\vee]$ is a polynomial algebra with two generators, say $v_1,v_2$ such that $S[E^\vee]\cong \z/2\z[v_1,v_2]$. Here the ring of invariants $S[E^\vee]^{\GL(E)}\cong\z/2\z[v_1,v_2]^{\GL(2,2)}$ is generated by
\begin{align*}
d_{1}&=v_1^2+v_2^2+v_1v_2,\\
d_{2}&=v_1v_2(v_1+v_2).
\end{align*}
Theorem \ref{thm.even} gives:
For $G=\SL(2,4)$, the total SWC of a non-trivial irreducible representation $\pi$ of $G$ is $$w(\pi)=1+d_{1}+d_{2}.$$

Now let $\mc D=\sum\limits_{i=1}^{r}d_{i}$ be the sum of the Dickson invariants, so that we can succinctly write
\beq
w(\pi)=(1+\mc D)^{m_\pi},
\eeq 
which is Theorem \ref{prop.two} from the Introduction.

\begin{ex} We have
\beq
w(\reg(G))=(1+ \mc D)^{q^2-1}.
\eeq
\end{ex}

\subsection{Corollaries}
Now we prove the corollaries mentioned in the Introduction.

\begin{c.intro}
  Let $G=\SL(2,q)$ with even $q$. Then the image of $w$ in \eqref{extend} is 
$$
  \{ (1+\mc D)^n \mid n \in \Z\}.$$
\end{c.intro}
\begin{proof}
Let $\pi$ be a non-trivial irreducible representation of $G$. Then, its total SWC is
$$w(\pi)=1+\mc D,$$
since $m_\pi=1$ by Lemma \ref{mpi}.
With the virtual representation $n\pi$, we obtain\\
 $(1+\mc D)^n$ in the image of $w$ for each $n\in\z$.
\end{proof}
The above proof also gives:
\begin{c.intro}
  Let $G=\SL(2,q)$ with $q=2^r$. Then, $$H^*_{\SW}(G)=\z/2\z[d_{1},\hdots,d_{r}].$$
\end{c.intro}

\begin{c.intro}\label{eceven}
Let $G=\SL(2,q)$ with $q$ even. Let $(\pi,V)$ be a representation of $G$. Then, $w_{\Top}(\pi)\neq 0$ if and only if $V^N=\{0\}$.
\end{c.intro}

\begin{proof}
From the preceding, we have
\beq
w_{\Top}(\reg(N)')=d_r,
\eeq
since this invariant has degree $q-1$.  The highest degree term of $w(\pi)$ is therefore
\beq
d_r^{m_\pi},
\eeq
which has degree $m_\pi (q-1)$; this is the degree of $\pi$ iff $\ell_\pi=0$ iff $V^N=\{0 \}$.
\end{proof}

\begin{remark} A representation $\pi$ of $\SL(2,q)$ is cuspidal iff $V^N=\{0 \}$. (See  \cite[page 410]{Bump}.)
 \end{remark}

\begin{c.intro}\label{obseven}
Let $G=\SL(2,q)$ with $q=2^r$. Let $\pi$ be a representation of $G$. Put $s=\ord_2(m_\pi)$. Then, the obstruction class of $\pi$ is equal to $w_{2^{r+s-1}}(\pi)=d_{1}^{2^s}.$
\end{c.intro}
\begin{proof}
From Theorem \ref{prop.two}, we have
$$w(\pi)=\sum_{i=0}^{m_\pi}{m_\pi\choose i}\mc D^i.$$
As in Corollary \ref{obs}, we can say that $m_\pi\choose2^s$ is the first odd binomial coefficient appearing in the above sum.  By expanding the term
$${m_\pi\choose 2^s}\mc D^{2^s}={m_\pi\choose 2^s}(d_{1}^{2^s}+\hdots+d_{r-1}^{2^s}+d_{r}^{2^s}),$$
we can deduce that ${m_\pi\choose 2^s}d_{1}^{2^s}$ has the least degree, which is $(2^{r-1}\cdot 2^s)$. Therefore, the least $k>0$ such that $w_k(\pi)\neq 0$ is $2^{r+s-1}$ as claimed. 
\end{proof}


\begin{thebibliography}{0}
\bibitem{Milgram}
	A. Adem and R.J.  Milgram.
 \textit{Cohomology of Finite Groups}, Vol. 309, Grundlehren der Mathematischen Wissenschaften (Springer-Verlag, Berlin, 2004). \url{https://doi.org/10.1007/978-3-662-06280-7}.

\bibitem{Bump}
D. Bump.
\textit{Automorphic Forms and Representations},  Cambridge Studies in Advanced Mathematics 55 (Cambridge University Press, 1998).

 \bibitem{Benson}
D.J. Benson.
\textit{Representations and cohomology II: Cohomology of groups and modules}, Cambridge Studies in Advanced Mathematics 31 (Cambridge University Press, 1998).
	
\bibitem{BrokerDieck}
T. Br\"ocker and T. T. Dieck.
\textit{Representations of Compact Lie Groups}, Vol. 98, Graduate Texts in Mathematics (Springer-Verlag, New York, 1995).

\bibitem{Bushnell.Henniart}
C.J. Bushnell and G. Henniart.
\textit{The Local Langlands Conjecture for {$\rm GL(2)$}}, Vol. 335, Grundlehren der Mathematischen Wissenschaften  (Springer-Verlag, Berlin, 2006). \url{https://doi.org/10.1007/3-540-31511-X}.
	
\bibitem{fiedorowicz2006homology}
Z. Fiedorowicz and S. Priddy. \textit{Homology of Classical Groups Over Finite fields and their Associated Infinite Loop Spaces},  Vol. 674, Lecture Notes in Mathematics (Springer, 2006).


\bibitem{fine1947binomial}
N.J. Fine. Binomial Coefficients Modulo a Prime.
\textit{The American Mathematical Monthly} 54, no. 10 (1947), 589--592.

\bibitem{PG}
P. Guillot. The computation of Stiefel-Whitney classes. \textit{Annales de l'Institut Fourier}, Vol. 60, no. 2 (2010), 565-606.

\bibitem{Fulton.Harris}
W. Fulton and J. Harris.
\textit{Representation theory, A First Course},  Vol. 129, Graduate Texts in Mathematics Readings in Mathematics (Springer-Verlag, New York, 1991). \url{https://doi.org/10.1007/978-1-4612-0979-9}.
	


\bibitem{GJgl2}
J. Ganguly and R. Joshi. Stiefel Whitney classes for real representations of $\GL_2 (\mathbb{F}_q)$.
	\textit{Internat. J. Math.} 33, no. 01 (2022), 2250010.

\bibitem{GJgln}
J. Ganguly and R. Joshi. Total Stiefel-Whitney classes for real representations of $\GL_n $ over $ \mathbb{F}_q,\mathbb{R} $ and $\mathbb{C} $.
\textit{Res. Math. Sci.} to appear.

\bibitem{Ganguly}
J. Ganguly and S. Spallone. Spinorial representations of symmetric groups.
\textit{J. Algebra} 544 (2020), 29--46.

\bibitem{gorenstein1980finite}
D. Gorenstein.
\textit{Finite groups}, 2nd ed. (Chelsea Publishing Company, New York, 1980).

\bibitem{gow1985real}
R. Gow. Real Representations of the Finite Orthogonal and Symplectic Groups of Odd Characteristic.
	\textit{J. Algebra} 96, no. 1 (1985), 249--274.

\bibitem{GKT}
J. Gunawardena, B. Kahn and C. Thomas. Stiefel-Whitney Classes of Real Representations of Finite Groups.
\textit{J. Algebra} 126, no. 2 (1989), 327--347. \url{https://doi.org/10.1016/0021-8693(89)90309-8}.

\bibitem{KT}
F. Kamber and Ph. Tondeur. Flat Bundles and Characteristic Classes of Group-Representations.
	\textit{Amer. J. Math.} 89, no. 4 (1967), 857--886. \url{https://doi.org/10.2307/2373408}.
	
\bibitem{Malik.thesis}
N. Malik.
Stiefel-Whitney Classes of Representations of Some Finite Groups of Lie Type.
Ph.D. thesis. Indian Institute of Science Education and Research Pune (2022).

\bibitem{may1999concise}
J.P. May. \textit{A Concise Course in Algebraic Topology}, Chicago Lectures in Mathematics (University of Chicago Press, 1999).


\bibitem{milnor}
J. Milnor and J. D. Stasheff.
\textit{Characteristic Classes} AM-76 (Princeton university press, 2016).

\bibitem{quillen}
D. Quillen. On the cohomology and {$K$}-theory of the general linear groups over a finite field.
\textit{Ann. of Math.} Second Series 96, no.3 (1972), 552--586.

\bibitem{vinroot}
C.R. Vinroot. Real Representations of Finite Symplectic Groups over Fields of Characteristic Two.
\textit{Int. Math. Res. Not. IMRN} 2020, no. 5 (2020), 1281--1299. \url{https://doi.org/10.1093/imrn/rny047}.
	
\bibitem{weibel2013k}
C.A. Weibel.
\textit{The K-book: An Introduction to Algebraic K-theory}, Vol. 145, Graduate Studies in Mathematics (American Mathematical Society Providence, RI, 2013).

\bibitem{wilkerson1983primer}
C. Wilkerson. A Primer on the Dickson Invariants.
	\textit{Amer. Math. Soc. Contemp. Math. Series} 19 (1983), 421--434.
	

\end{thebibliography}
 \end{document}